\documentclass[a4paper,12pt,reqno,final]{amsart}
\usepackage{version}
\usepackage[notcite,notref]{showkeys}
\usepackage{amsmath,amsthm,amssymb,latexsym,epsfig,graphicx,subfigure}
\usepackage{color}
\newtheorem{theorem}{Theorem}[section]
\newtheorem{lemma}[theorem]{Lemma}
\newtheorem{corollary}[theorem]{Corollary}

\theoremstyle{definition}
\newtheorem{definition}[theorem]{Definition}
\newtheorem{example}[theorem]{Example}

\newtheorem{remark}[theorem]{Remark}

\DeclareMathOperator{\dimH}{dim_H}
\DeclareMathOperator{\dimb}{dim_B}
\DeclareMathOperator{\dimp}{dim_p}
\DeclareMathOperator{\udimb}{\overline{dim}_B}

\DeclareMathOperator{\Per}{Per}
\DeclareMathOperator{\sgn}{sgn}

\numberwithin{equation}{section}

\newcommand{\ba}{{\bf a}}
\newcommand{\bg}{{\bf g}}
\newcommand{\bi}{{\bf i}}
\newcommand{\bj}{{\bf j}}
\newcommand{\bk}{{\bf k}}
\newcommand{\bv}{{\bf v}}
\newcommand{\bw}{{\bf w}}
\newcommand{\bu}{{\bf u}}
\newcommand{\bz}{{\bf z}}

\begin{document}

\title[Falconer-Sloan condition]{Random affine code tree fractals and 
Falconer-Sloan condition}

\author[E. J\"arvenp\"a\"a]{Esa J\"arvenp\"a\"a}
\address{Department of Mathematical Sciences, P.O. Box 3000,
  90014 University of Oulu, Finland}
\email{esa.jarvenpaa@oulu.fi}

\author[M. J\"arvenp\"a\"a]{Maarit J\"arvenp\"a\"a}
\address{Department of Mathematical Sciences, P.O. Box 3000,
  90014 University of Oulu, Finland}
\email{maarit.jarvenpaa@oulu.fi}

\author[B. Li]{Bing Li}
\address{Department of Mathematics, South China University of Technology,
Guangzhou, 510641, P.R. China}
\address{Department of Mathematical Sciences, P.O. Box 3000,
  90014 University of Oulu, Finland}
\email{libing0826@gmail.com, the corresponding author}

\author[\"O. Stenflo]{\"Orjan Stenflo}
\address{Department of Mathematics, Uppsala University, P.O. Box 480,
  75106 Uppsala, Sweden}
\email{stenflo@math.uu.se}

\thanks{We thank the referee for useful comments and we acknowledge the support 
of Academy of Finland, the Centre of
Excellence in Analysis and Dynamics Research. BL is partially supported by
a NSFC grant 11201155. \"OS thanks the Esseen foundation.}

\begin{abstract}
We calculate the almost sure dimension for a general class of random affine 
code tree fractals in $\mathbb R^d$. The result is based on a probabilistic
version of the Falconer-Sloan condition $C(s)$ introduced in
\cite{FS}. We verify that, in general, systems having a small number of maps
do not satisfy condition $C(s)$. However, there exists a natural number $n$
such that for typical systems the family of all iterates up to level $n$
satisfies condition $C(s)$. 
\end{abstract}

\maketitle

\section{Introduction}\label{intro}

In the investigation of dimensional properties of self-similar and
self-conformal sets an important tool is the thermodynamic formalism.
There is a natural way to attach a pressure function to a self-similar or
self-conformal iterated function system and, for example, the Hausdorff
dimension and multifractal spectrum can be calculated using the pressure.
Since the pressure is defined by an additive potential
function, there are many tools available for the purpose of analysing it.

In his famous theorem from 1988 Falconer \cite{F88} proved that
the dimension of any typical self-affine set is equal to the unique zero of the
pressure function under the assumption that the norms of the linear parts
are less than 1/3. Later Solomyak \cite{S} verified that 1/3 can be replaced
by 1/2 which is the best possible bound, see \cite{PU}. The potential is
defined by means of the singular value functions of the
iterates of the linear parts, and contrary to the self-conformal setting,
the potential $\phi$ is not additive. In the self-affine case $\phi$
is subadditive guaranteeing the existence of the pressure and
its unique zero. However, $\phi$ is not superadditive -- not even in the weak
sense that $\phi(n+m)\ge\phi(n)+\phi(m)-C$ for some constant $C$.
In many cases this causes severe problems, see for example \cite{BF},
\cite{FM07}, \cite {FS}, \cite{Fe}, \cite {FeSh}, \cite{JJKKSS} and \cite{KS}.

There are various ways to introduce randomness to the self-affine setting.
In \cite{JPS} Jordan, Pollicott and Simon considered a fixed affine
iterated function system  with a small random perturbation in translations at
each step of the construction. When investigating random subsets of
self-affine attractors, Falconer and Miao \cite{FM09} selected at each step
of the construction a random subfamily of the original function system
independently. Both in \cite{JPS} and \cite{FM09} there is total independence
both in space, that is, between different nodes at a fixed construction level,
and in scale, meaning that once a node is chosen its descendants are
chosen independently of the previous history. Such systems are called 
statistically self-affine, since the law controlling the construction is the
same at every node. However, typical realisations are not self-affine.
Inspired by the random $V$-variable fractals introduced by Barnsley, Hutchinson
and Stenflo in \cite{BHS2005}, a new class of random self-affine code tree 
fractals was proposed in \cite{JJKKSS}. In this class typical realisations 
mimic the self-affinity of deterministic iterated function systems. 
Moreover,
the probability distributions have certain independence only in
scale,  and therefore, typical realisations are locally random 
but globally nearly homogeneous. In particular, the attractor is a finite union
of self-affine copies of sets with arbitrarily small diameter. Thus
typical realisations are close to deterministic self-affine 
sets.  In a code tree fractal the linear parts of
the iterated function system may depend on the construction step. 
For example, attractors of graph directed Markov
systems generated by affine maps \cite{F}, or more generally sub-self-affine
sets \cite{F95}, are code tree fractals.

In this paper we generalise the dimension
results in \cite{JJKKSS} concerning random affine code tree fractals.
In \cite{JJKKSS} the existence of the pressure was proven under quite general
conditions (see Theorem~\ref{pexists}). However, when verifying the relation
between the dimension and the zero of the pressure several additional
assumptions were needed -- the most restrictive one being that $d=2$. The main
cause for the extra assumptions was the non-superadditivity of the potential
defining the pressure.
In the self-affine setting various approaches have been introduced
to overcome the problems caused by the non-superadditivity of the potential.
These include the cone condition \cite{BF}, \cite{FM07}, \cite{FeSh},
\cite{KS}, irreducibility \cite{Fe} and non-existence of parallelly
mapped vectors \cite{JJKKSS}. In this paper we focus on a general condition
(see Definition~\ref{FSdef}) introduced recently by Falconer and Sloan
\cite{FS}. Under the Falconer-Sloan condition (for brevity, F-S condition) 
higher dimensional spaces can also be
considered, see Theorem~\ref{maintheorem}. The only additional assumption
compared to Theorem~\ref{pexists} is that some iterates of the system satisfy
the F-S condition with positive probability. 

The F-S condition is related to a family of linear maps on 
$\mathbb R^d$. The condition is open in the sense that the set of families of
linear maps satisfying it is open in any natural topology.
In this paper we also address a problem proposed by Falconer concerning the
genericity of the F-S condition. In $\mathbb R^2$ the F-S condition 
is easy to check but in higher dimensional spaces the question is more delicate.
It turns out that a family of linear maps $\{S_i\}_{i=1}^k$ on $\mathbb R^d$
does not satisfy the F-S condition unless $k$ is sufficiently large
(see Remark~\ref{FSremark}.(b)) -- the minimal value of $k$ being much larger
than $d$. However, in Corollary~\ref{typical} we prove that there
exists a natural number $n$ depending only on $d$ such that for any generic
family $\{S_i\}_{i=1}^k$ the family
$\{S_{i_1}\circ\dots\circ S_{i_l}\mid i_j\in\{1,\dots,k\}\text{ for }
  j=1,\dots,l\text{ and } 1\le l\le n\}$
satisfies the F-S condition. The set is generic
both in the topological sense, that is, it is open and dense, and in the
measure theoretic sense meaning that it has full Lebesgue measure.
Theorem~\ref{CsCm} provides an explicit criterion guaranteeing that a family
$\{S_i\}_{i=1}^k$ belongs to the generic set. In Remark~\ref{suffcon} we explain
why the complement of this generic set is non-empty, that is, why 
Corollary~\ref{typical} is not valid for all families. 

In many problems related to self-affine iterated function systems it is
sufficient to study iterates of the maps. This is also the case in 
Theorem~\ref{maintheorem}. The applicability of the
F-S condition is based on the fact that the upper bound $n$ for the
number of iterates needed in order that the family
$\{S_{i_1}\circ\dots\circ S_{i_l}\mid i_j\in\{1,\dots,k\}\text{ for }
  j=1,\dots,l\text{ and } 1\le l\le n\}$
satisfies the F-S condition is a constant depending only on the
dimension of the ambient space. In particular, Corollary~\ref{typical} implies
that typical systems satisfy the assumptions of Theorem~\ref{maintheorem}.

The paper is organised as follows. In Section \ref{FS} we recall the
Falconer-Sloan setting and prove that the F-S condition is valid for a family
of iterates of a generic family (Corollary~\ref{typical}). Moreover, we give an
explicit criterion implying that a family belongs to this generic set
(Theorem~\ref{CsCm}). In Section~\ref{codetree} we recall the notation from
\cite{JJKKSS} concerning random affine code tree fractals and prove
that the dimension of a typical affine code tree fractal is given by the zero
of the pressure (Theorem~\ref{maintheorem}).

\section{Falconer-Sloan condition $C(s)$}\label{FS}

In this section we consider the genericity of the F-S condition
introduced in \cite{FS} for the purpose of overcoming problems caused
by the fact that in the self-affine setting the natural potential defining the 
pressure (for definition see \eqref{pressure}) is not supermultiplicative.
Intuitively, the reason behind the applicability of the F-S
condition is as follows: Letting $A$ and $B$ be $d\times d$-matrices, the norm
$\Vert AB\Vert$ may be much smaller than $\Vert A\Vert\cdot\Vert B\Vert$.
This happens if the vector $v$ which determines the norm of $B$ is mapped by
$B$ onto an eigenspace of $A$ which corresponds to some small eigenvalue of
$A$. In the expression of the pressure (for $s=1$) there is a sum 
of 
terms
of the form $\Vert A B\Vert$. The F-S condition guarantees that  
$\Vert A B\Vert$ is not much less than $\Vert A\Vert\cdot\Vert B\Vert$ 
simultaneously for all pairs $(A,B)$.

We begin by recalling the notion from \cite{FS}. For all $m\in\mathbb N$ with
$0\le m\le d$ we denote by $\Lambda^m$ the $m$-th exterior power of
$\mathbb R^d$ with the convention $\Lambda^0=\mathbb R$. An
$m$-vector $\bv\in\Lambda^m$ is {\it decomposable} if it can be written as
$\bv=v_1\wedge\dots\wedge v_m$ for some $v_1,\dots,v_m\in\mathbb R^d$. Let
$\Lambda_0^m$ be the set of decomposable $m$-vectors. If $\{e_1,\dots,e_d\}$
is a basis of $\mathbb R^d$, then
$\{e_{i_1}\wedge\dots\wedge e_{i_m}\mid 1\le i_1<\dots<i_m\le d\}$ is a basis of
$\Lambda^m$. Supposing that
$\{e_1,\dots,e_d\}$ is an orthonormal basis of $\mathbb R^d$, 
{\it the Hodge star operator} $*:\Lambda^m\to\Lambda^{d-m}$ is defined as the 
linear map satisfying
\[
*(e_{i_1}\wedge\dots\wedge e_{i_m})=e_{j_1}\wedge\dots\wedge e_{j_{d-m}}
\]
for all $1\le i_1<\dots<i_m\le d$, where $1\le j_1<\dots<j_{d-m}\le d$ satisfy
$\{i_1,\dots,i_m\}\cup\{j_1,\dots,j_{d-m}\}=\{1,\dots,d\}$. Let
$\omega=e_1\wedge\dots\wedge e_d$ be the normalised volume form on
$\mathbb R^d$. Recall that $\Lambda^d$ is one dimensional. We define the inner
product $\langle\cdot\mid\cdot\rangle$ on $\Lambda^m$ by the (implicit) formula
\[
\langle\bv\mid\bw\rangle\omega=\bv\wedge*\bw.
\]
Then the inner product is independent of the choice of the 
orthonormal basis 
$\{e_1,\dots,e_d\}$, and moreover,
$\{e_{i_1}\wedge\dots\wedge e_{i_m}\mid 1\le i_1<\dots<i_m\le d\}$
becomes an orthonormal basis of $\Lambda^m$. Any linear map
$S:\mathbb R^d\to\mathbb R^d$ induces a linear map $S:\Lambda^m\to\Lambda^m$
such that $S(v_1\wedge\dots\wedge v_m)=Sv_1\wedge\dots\wedge Sv_m$ for all
$v_1\wedge\dots\wedge v_m\in\Lambda_0^m$.

Now we are ready to recall the definition of the condition $C(s)$ from
\cite{FS} -- first for integer parameters and after that for non-integral
parameters $s$.

\begin{definition}\label{FSdef}
Consider a family $\{S_i:\mathbb R^d\to\mathbb R^d\}_{i\in I}$ consisting of
linear maps. Let $m\in\mathbb N$ with $0\le m\le d$. The family
$\{S_i\}_{i\in I}$ satisfies {\it condition $C(m)$} if for all
$0\ne\bv,\bw\in\Lambda_0^m$ there is $i\in I$ such that
$\langle S_i\bv\mid\bw\rangle\ne 0$. Let $0<s<d$ be non-integral and
let $m$ be the integer part of $s$. The family $\{S_i\}_{i\in I}$ satisfies
{\it condition $C(s)$} if for all
$0\ne\bv,\bw\in\Lambda_0^m$ and $0\ne\bv\wedge v,\bw\wedge w\in\Lambda_0^{m+1}$
there is $i\in I$ such that
$\langle S_i\bv\mid\bw\rangle\ne 0$ and
$\langle S_i(\bv\wedge v)\mid\bw\wedge w\rangle\ne 0$.
\end{definition}

\begin{remark}\label{FSremark}
(a) The family $\{S_i\}_{i\in I}$ satisfies condition $C(m)$ if and only if for
all $0\ne\bv\in\Lambda_0^m$ the set $\{S_i\bv\mid i\in I\}$ spans $\Lambda^m$.
Here the if-part is clear whereas the only if -part involves a slight subtilty.
Indeed, Definition \ref{FSdef} deals with decomposable vectors and
$\Lambda_0^m$ is not a vector space when $m\notin\{0,1,d-1,d\}$.
For the only if -part, assume that there exists $0\ne\bv\in\Lambda_0^m$
such that the set $\{S_i\bv\mid i\in I\}$ does not span $\Lambda^m$. Letting
$k$ be the maximal number of linearly independent vectors in
$\{S_i\bv\mid i\in I\}$, we have $k<\binom dm=\dim\Lambda^m$. Denote these
vectors by $\bw^1,\dots,\bw^k$ and consider $i=1,\dots,k$.
Now $\bu=u_1\wedge\dots\wedge u_m\in\Lambda_0^m$
is perpendicular to $\bw^i=w_1^i\wedge\dots\wedge w_m^i$, if and only if the
vectors $P_iu_1,\dots,P_iu_m$ are linearly dependent. Here $P_i$ is the
orthogonal projection onto the $m$-dimensional linear subspace spanned by
$w_1^i,\dots,w_m^i$. Using the notation $B$ for the $m\times m$-matrix whose
columns are the vectors $P_iu_1,\dots,P_iu_m$ expressed in the basis
$\{w_1^i,\dots,w_m^i\}$, we observe that the vectors $P_iu_1,\dots,P_iu_m$ are
linearly dependent, if and only if the determinant of $B$ is zero.
This implies the existence of a polynomial map $Q:\mathbb R^{d^m}\to\mathbb R$
such that $\langle\bu\mid\bw^i\rangle=0$ if and only if $Q(u_1,\dots,u_m)=0$.
This, in turn, gives  that for all $i=1,\dots,k$ the set
\[
M_i=\{(u_1,\dots,u_m)\in\mathbb R^{d^m}\mid\langle\bu\mid\bw^i\rangle=0\}
\]
has codimension 1, and clearly, $0\in M_i$. Note that
$\bu=u_1\wedge\dots\wedge u_m=0$ if and only if
the vectors $u_1,\dots,u_m$ are linearly dependent, that is, all the
$m\times m$-minors are zero for the $d\times m$-matrix whose columns are the
vectors $u_1,\dots,u_m$. Since there are $\binom dm$ such minors  and 
$k<\binom dm$, there exists
$\overline\bu=(\overline u_1,\dots,\overline u_m)\in\cap_{i=1}^kM_i$ such that
$\overline\bu\ne 0$. In particular, $\langle\overline\bu\mid\bw^i\rangle=0$
for all
$i=1,\dots,k$.
Therefore condition $C(m)$ is not satisfied.

(b) From (a) we see that there must be at least  $\binom dm$ maps in the
family $\{S_i\}_{i\in I}$ for condition $C(m)$ to be satisfied. Note that when
$d$ is large and $1<m<d-1$ the number $\binom dm$ is much larger than $d$.

(c) If $m<s<m+1$ and $\{S_i\}_{i\in I}$ satisfies condition $C(s)$ then it
satisfies condition $C(t)$ for all $m\le t\le m+1$. In \cite[Lemma 2.6]{FS} 
it is shown that the irreducibility condition used
by Feng in \cite{Fe} is (essentially) equivalent to the condition $C(1)$.
\end{remark}

We proceed by introducing the notation needed for studying the validity of the
F-S condition. Let $F,G:\mathbb R^d\to\mathbb R^d$ be linear mappings with $d$ 
different real eigenvalues $\{\lambda_1,\dots,\lambda_d\}$ and 
$\{t_1,\dots,t_d\}$, respectively. Let $\{\hat e_1,\dots,\hat e_d\}$ and
$\{\tilde e_1,\dots,\tilde e_d\}$
be the corresponding normalised eigenvectors. We assume that for all
$k=1,\dots,d$
\begin{equation}\label{eigenspacecondition}
\begin{split}
&\lambda_{i_1}\cdots\lambda_{i_k}\ne\lambda_{j_1}\cdots\lambda_{j_k}\text{ and }
 t_{i_1}\cdots t_{i_k}\neq t_{j_1}\cdots t_{j_k}\text{ for all pairs }\\
& (i_1,\dots,i_k)\neq (j_1,\dots,j_k).
\end{split}
\end{equation}
Let $A=A(F,G):\mathbb R^d\to\mathbb R^d$ be the linear map satisfying
$\tilde e_i=A^{-1}e_i$, that is, $e_i=A\tilde e_i$ for all $1\le i\le d$. 
Let $\mathcal S_k=\mathcal S_k(F,G)$
be the family of compositions of $F$ and $G$ up to level $k$, that is,
\begin{equation}\label{Skdef}
\mathcal S_k=\{T_1\circ\cdots\circ T_j\mid 1\le j\le k\text{ and }
T_i\in\{F,G\} \text{ for all } 1\le i\le j\}.
\end{equation}
Using the eigenbasis $\{\hat e_1,\dots,\hat e_d\}$ of $F$ as the 
basis of $A$, we view 
$A$ as a $d\times d$-matrix. Denote by $\mathcal M_d$ the class of
$d\times d$-matrices whose minors are all non-zero.

With the above notation we prove two lemmas.

\begin{lemma}\label{takingall}
Let $1\le m\le d$ and $A\in\mathcal M_d$ be as above. For all 
$1\le i_1<\cdots<i_m\le d$ write
\begin{equation}\label{representation}
\hat e_{i_1}\wedge\cdots\wedge\hat e_{i_m}=\sum_{1\le j_1<\cdots<j_m\le d}
 c_{i_1\cdots i_m}^{j_1\cdots j_m}\,\tilde e_{j_1}\wedge\cdots\wedge\tilde e_{j_m}.
\end{equation}
Then $c_{i_1\cdots i_m}^{j_1\cdots j_m}\neq 0$ for all $(i_1,\dots, i_m)$ and
$(j_1,\dots, j_m)$.
\end{lemma}

\begin{proof}
We denote the set of all permutations of $(j_1,\dots,j_m)$ by 
$\Per(j_1,\dots,j_m)$ and write $\sgn(\sigma)$ for the sign of a permutation 
$\sigma\in\Per(j_1,\dots,j_m)$.
Since $\hat e_{i_l}=\sum_{j=1}^dA_{j i_l}\tilde e_j$ for all $1\le l\le m$ and
the wedge product is antisymmetric and multilinear, we have
\begin{align*}
\hat e_{i_1}\wedge\cdots\wedge\hat e_{i_m}&=\sum_{j_1=1}^d\dots\sum_{j_m=1}^d 
 A_{j_1i_1}\cdots A_{j_mi_m}\tilde e_{j_1}\wedge\cdots\wedge\tilde e_{j_m}\\ 
&=\sum_{1\le j_1<\cdots<j_m\le d}\bigl(\sum_{\sigma\in\Per(j_1,\dots,j_m)}
 \sgn(\sigma)A_{\sigma_1i_1}\cdots A_{\sigma_mi_m}\bigr)\tilde e_{j_1}\wedge\cdots
 \wedge\tilde e_{j_m}\\
&=c_{i_1\cdots i_m}^{j_1\cdots j_m}\,\tilde e_{j_1}\wedge\cdots\wedge\tilde e_{j_m}.
\end{align*}
Thus the coefficient
$c_{i_1\cdots i_m}^{j_1\cdots j_m}$ is the minor of $A$ determined by the
columns $i_1,\dots,i_m$ and rows $j_1,\dots,j_m$, and by the definition of
$\mathcal M_d$, we have $c_{i_1\cdots i_m}^{j_1\cdots j_m}\ne 0$.
\end{proof}

For all $1\le m\le d$, define
\[
B_1=\{\hat e_{i_1}\wedge\cdots\wedge\hat e_{i_m}\mid 1\le i_1<\dots<i_m\le d\}
\]
and
\[
B_2=\{\tilde e_{j_1}\wedge\cdots\wedge\tilde e_{j_m}\mid 1\le j_1<\dots<j_m
  \le d\}.
\]
Then $B_1$ and $B_2$ are bases of $\Lambda^m$. Furthermore, the elements of
$B_1$ and $B_2$ are the eigenvectors of $F:\Lambda^m\to\Lambda^m$ and
$G:\Lambda^m\to\Lambda^m$ with eigenvalues
$\lambda_{i_1}\cdots\lambda_{i_m}$ and $t_{j_1}\cdots t_{j_m}$, respectively.

\begin{remark}\label{Vandermonde}
Let $a_1,\dots,a_d\in\mathbb R\setminus\{0\}$ with $a_i\ne a_j$ for $i\ne j$
and let $v=(v_1,\dots,v_d)\in\mathbb R^d$ with $v_i\ne 0$ for all
$i=1,\dots,d$. Denoting by $(a_j)^i$ the $i$-th power of $a_j$, it
follows from the Vandermonde determinant formula that the vectors
$\{((a_1)^iv_1,\dots,(a_d)^iv_d)\mid i=k,\dots,k+d-1\}$ span $\mathbb R^d$ for
all $k\in\mathbb N$.
By induction it is easy to see that the vectors
\[
\{v^{i_j}=((a_1)^{i_j}v_1,\dots,(a_d)^{i_j}v_d)\mid j=1,\dots,d\text{ and }
i_1<\dots<i_d\}
\]
span $\mathbb R^d$. Indeed, the case $d=1$ is obvious. Assuming that the claim
is true for $d$, we show that the vectors $\{v^{i_1},\dots,v^{i_{d+1}}\}$ span
$\mathbb R^{d+1}$. Suppose to the contrary that this is not the case,
that is, there is $j$ such that $v^{i_j}=\sum_{k\ne j}\alpha_kv^{i_k}$.
For all $k=1,\dots,d+1$ we denote by $\Pi_k:\mathbb R^{d+1}\to\mathbb R^d$ the
projection which omits the $k^{\text{th}}$ coordinate. Fix $l\ne j$. Now the
induction hypothesis implies that
$\Pi_jv^{i_j}=\sum_{k\ne j}b_k\Pi_jv^{i_k}$ and
$\Pi_lv^{i_j}=\sum_{k\ne l}c_k\Pi_lv^{i_k}$ where the coefficients $b_k$ and
$c_k$ are unique. Since $a_l\ne a_j$, we have $b_k\ne c_k$ for some $k\ne j$.
On the other hand,
$\Pi_jv^{i_j}=\sum_{k\ne j}\alpha_k\Pi_jv^{i_k}$ and
$\Pi_lv^{i_j}=\sum_{k\ne l}\alpha_k\Pi_lv^{i_k}$, and therefore, 
$\alpha_k=b_k=c_k$ for all $k\ne j$ which is a contradiction.
\end{remark}

\begin{lemma}\label{nonzerocoordinates}
Let $0\ne\bv\in\Lambda^m$ and $n=\binom dm$. Then there are at most $n(n-1)$
numbers $i\in\mathbb N$ with the property that at least one coordinate
of $F^i\bv$ with respect to the basis $B_2$ is equal to zero.
\end{lemma}

\begin{proof} Let $\bv=(v_1,\dots,v_n)$ be the coordinates of $\bv$ with respect
to the basis $B_1$ and let $k$ be the number of non-zero coordinates.
We denote by $V_{\bv}$ the $k$-dimensional plane spanned by those basis vectors
in $B_1$ that correspond to the non-zero coordinates of $\bv$. Let
$\gamma_1,\dots,\gamma_n$ be the eigenvalues of $F:\Lambda^m\to\Lambda^m$.
Observe that for the $i$-th  iterate $F^i$ of $F$ we have
$F^i\bv=(\gamma_1^iv_1,\dots,\gamma_n^iv_n)$. Combining
\eqref{eigenspacecondition} with Remark~\ref{Vandermonde},
implies that the set $\{F^{i_1}\bv,\dots,F^{i_k}\bv\}$ spans $V_{\bv}$ for all
natural numbers $i_1<i_2<\cdots<i_k$. For all $j=1,\dots,n$, let
\[
W_j=\{\bw\in\Lambda^m\mid\bw=(w_1,\dots,w_n)\text{ with respect to }B_2
\text{ and }w_j=0\}.
\]
Applying Lemma~\ref{takingall} gives for all $j=1,\dots,n$ and
$1\le i_1<\dots<i_m\le d$ that 
$\hat e_{i_1}\wedge\cdots\wedge\hat e_{i_m}\notin W_j$.
Thus the dimension of $V_\bv\cap W_j$ is strictly less than $k$. We conclude
that for all $j=1,\dots,n$, there are at most $k-1$ indices $i$ such that
$F^i\bv\in W_j$, and therefore, there are at most $n(k-1)$ indices $i$ such
that $F^i\bv\in W_j$ for some $j=1,\dots,n$. Since this is true for all
$1\le k\le n$, the claim follows.
\end{proof}

Now we are ready to prove our main theorem in this section. For this
purpose, set $n_0=\max\limits_{0\le m\le d}\binom dm$. After proving 
Corollary~\ref{typical}, we discuss the criterion which is based on the 
following theorem and gives a sufficient condition for the validity of the 
F-S condition (see Remark~\ref{suffcon}).

\begin{theorem}\label{CsCm}
Let $F$ and $G$ be as in \eqref{eigenspacecondition} and assume
that
$A=A(F,G)\in\mathcal M_d$. Then the family $\mathcal S_{2n_0^2}$ 
defined in
\eqref{Skdef} satisfies the condition $C(s)$ for all $0\le s\le d$.
\end{theorem}

\begin{proof}
By Remark~\ref{FSremark} (c) it is enough to prove that the family
$\mathcal S_{2n_0^2}$ satisfies the condition $C(s)$ for non-integral $s$.
Letting $m$ be the integer part of $s$, set $n_1=\binom dm$ and
$n_2=\binom d{m+1}$ and define $M=n_1(n_1-1)+n_2(n_2-1)+1$ and $N=n_1+n_2-1$.
Let $0\neq\bv,\bw\in\Lambda^m$ and $0\neq\bu,\bz\in\Lambda^{m+1}$. By applying
Lemma~\ref{nonzerocoordinates} to the iterates $F^i\bv$ and $F^i\bu$, where
$1\le i\le M$, we deduce that there exists $1\le i_0\le M$ such that all
coordinates of the iterates $F^{i_0}\bv$ and $F^{i_0}\bu$ with respect to the
basis $B_2$ are non-zero. Furthermore, from Remark~\ref{Vandermonde} we see
that for all $j_1<\dots<j_{n_1}$ the vectors
$G^{j_1}(F^{i_0}\bv),\dots,G^{j_{n_1}}(F^{i_0}\bv)$ span $\Lambda^m$. Hence,
there are at least $N-n_1+1$ indices $j=1,\dots,N$ such that the points
$G^j(F^{i_0}\bv)$ do not belong to the orthogonal complement 
$\bw^\perp$ of $\bw$.
A similar argument implies that among these  $N-n_1+1$ indices there exists
$j_0$ such that $G^{j_0}(F^{i_0}\bu)\notin\bz^\perp$, and therefore,
\[
\langle G^{j_0}F^{i_0}\bv\mid\bw\rangle\ne 0\text{ and }
  \langle G^{j_0}F^{i_0}\bu\mid\bz\rangle\ne 0
\]
implying that $\mathcal S_{M+N}$ satisfies $C(s)$. Since $M+N\le 2n_0^2$ this
completes the proof of the claim.
\end{proof}

Let $k\in\mathbb N$. We identify the space of families
$\mathcal F=\{S_i:\mathbb R^d\to\mathbb R^d\}_{i=1}^k$ of linear maps
with $\mathbb R^{d^2k}$. For $\mathcal F\in\mathbb R^{d^2k}$ define
\[
\mathcal S_l(\mathcal F)=\{S_{i_1}\circ\cdots\circ S_{i_j}\mid 1\le j\le l
\text{ and }S_{i_m}\in\mathcal F\text{ for all } 1\le m\le j\}.
\]
With this notation we have the following consequence of Theorem~\ref{CsCm}.

\begin{corollary}\label{typical}
Letting $k\ge 2$ be a natural number, the set
\[
\mathcal C=
\{\mathcal F\in\mathbb R^{d^2k}\mid\mathcal S_{2n_0^2}(\mathcal F)
  \text{ satisfies }C(s)\text{ for all }0\le s\le d\}
\]
is open, dense and has full Lebesgue measure. More precisely, 
$\mathbb R^{d^2k}\setminus\mathcal C$ is contained in a finite union of 
$(d^2k-1)$-dimensional algebraic varieties.
\end{corollary}

\begin{proof}
We start with an easy observation: assuming that $\mathcal F\subset\mathcal G$
are families of linear maps on $\mathbb R^d$ and $\mathcal F$ satisfies
condition $C(s)$, then $\mathcal G$ satisfies it too. Thus it is enough to
prove the claim in the case $k=2$. 
The set of $d\times d$-matrices with
a fixed non-zero minor is a $(d^2-1)$-dimensional algebraic variety. 
Since the number of minors is finite, the set
$\mathbb R^{d^2}\setminus\mathcal M_d$ can be represented as a finite union of
$(d^2-1)$-dimensional algebraic varieties, implying that 
$\mathcal M_d\subset\mathbb R^{d^2}$ is open, dense and has full Lebesgue
measure. Moreover, note that the set of pairs
$(F,G)$ of linear maps having $d$ real eigenvalues and not satisfying 
\eqref{eigenspacecondition} is a finite union of $(2d^2-1)$-dimensional 
algebraic varieties. Thus the set of pairs $(F,G)$
satisfying the assumptions of Theorem~\ref{CsCm} is
open and has positive Lebesgue measure. For the purpose of verifying that
$\mathcal C$ is dense and has full  Lebesgue measure, we need to extend
our argument to the case where $F$ and $G$ are allowed to have complex
eigenvalues satisfying \eqref{eigenspacecondition}.

Recall that if $\lambda=re^{i\theta}$ is a complex eigenvalue of $F$, also
$\overline\lambda=re^{-i\theta}$ is an eigenvalue of $F$, and there is a two
dimensional invariant subspace $V\subset\mathbb R^d$ where $F$ acts as the
rotation by angle $\theta$ composed with scaling by $r$. Let
$e_1, e_2\in\mathbb R^d$ be such that $V$ is spanned by $e_1$ and $e_2$ and let
$e_3$ be an eigenvector of $F$ corresponding to a real eigenvalue $t$. Then
$e_3\wedge e_1$ and $e_3\wedge e_2$ span an eigenspace of $F$ on $\Lambda^2$
corresponding to the eigenvalue $t\lambda$.
If $\rho$ is another complex eigenvalue of $F$ and $e_4$ and $e_5$ span
the corresponding eigenspace, then $e_1\wedge e_2$ and $e_4\wedge e_5$ are
eigenvectors of $F$ on $\Lambda^2$ with eigenvalues $\lambda\overline\lambda$
and $\rho\overline\rho$,
respectively. The 4-dimensional subspace spanned by
$\{e_1\wedge e_4,e_1\wedge e_5,e_2\wedge e_4,e_2\wedge e_5\}$ is divided into
two invariant 2-dimensional subspaces corresponding to the complex
eigenvalues
$\lambda\rho$ and $\lambda\overline\rho$. By \eqref{eigenspacecondition},
the numbers $\lambda\overline\lambda,\rho\overline\rho,\lambda\rho$
and $\lambda\overline\rho$ are different. In this way we find a basis of
$\Lambda^m$ consisting of eigenvectors of $F$. Since the Vandermonde
determinant formula applies also for complex entries, 
Theorem \ref{CsCm} is valid for an open dense set of pairs of linear maps
$(F,G)$ having full Lebesgue measure. This completes the proof.
\end{proof}

\begin{remark}\label{suffcon} 
(a) Let $\mathcal F=\{T_i:\mathbb R^d\to\mathbb R^d\}_{i=1}^m$
be an iterated function system consisting of affine mappings
$T_i(x)=S_i(x)+a_i$. When considering the validity of the F-S
condition, the translation parts $a_i$ play no role. From Theorem~\ref{CsCm}
and  Corollary~\ref{typical} we conclude that if there are $i\ne j$ such that
the eigenvalues of $S_i$ and $S_j$ satisfy \eqref{eigenspacecondition} and
the eigenvectors of $S_i$ are mapped to those of $S_j$ by some 
$A\in\mathcal M_d$ then $\mathcal S_{2n_0^2}(\mathcal F)$ satisfies 
the condition $C(s)$ for all $0\le s\le d$.

(b) Let $\mathcal F$ be as in remark (a). If $\mathcal F$ is not irreducible,
that is, if there exists a non-trivial proper subspace $V\subset\mathbb R^d$
satisfying $S_i(V)\subset V$ for all $i=1,\dots,m$, then by 
Remark~\ref{FSremark}.(a) the family
$\mathcal S_N(\mathcal F)$ does not satisfy the condition $C(s)$ for any
$0<s<d$ and for any $N\in\mathbb N$. 
\end{remark}

\section{Random affine code tree fractals}\label{codetree}

In this section we consider the Falconer-Sloan setting for a class of
random affine code tree fractals introduced in \cite{JJKKSS} which are locally 
random but globally nearly 
homogeneous. It turns out that the earlier results in \cite{JJKKSS}
can be improved under a probabilistic version of the condition $C(s)$. We begin
by recalling the notation from \cite{JJKKSS}. 

Let 
$\mathcal F=\{F^\lambda=\{f_1^\lambda,\dots,f_{M_\lambda}^\lambda\}
\mid\lambda\in\Lambda\}$ be a family of iterated function systems on
$\mathbb R^d$.  Here the index set $\Lambda$ is a topological 
space. Assume
that for all $i=1,\dots,M_\lambda$ the maps
$f_i^\lambda\colon\mathbb R^d\to\mathbb R^d$ are affine, that is,
$f_i^\lambda(x)= T_i^\lambda(x)+a_i^\lambda$, where $T_i^\lambda$ is a non-singular
linear mapping and $a_i^\lambda\in\mathbb R^d$.
We consider the case where the norms and the numbers of the maps are uniformly
bounded meaning that
\[
\sup_{\lambda\in\Lambda,i=1,\dots,M_\lambda}\Vert T_i^\lambda\Vert<1,
\sup_{\lambda\in\Lambda,i=1,\dots,M_\lambda}|a_i^\lambda|<\infty\text{ and }
 M=\sup_{\lambda\in\Lambda}M_\lambda<\infty.
\]
Identifying $F^\lambda$ with an element of $\mathbb R^{(d^2+d)M_\lambda}$, gives
$\mathcal F\subset\bigcup_{i=1}^M\mathbb R^{(d^2+d)i}$, where the union is 
disjoint. We equip $\bigcup_{i=1}^M\mathbb R^{(d^2+d)i}$ with the natural 
topology and assume that $\lambda\mapsto F^\lambda$ is a Borel map. Similarly,
the linear parts $T_i^\lambda$ are embedded in $\mathbb R^{d^2M_\lambda}$.

We continue by introducing the concept of a code tree which is a modification 
of the standard tree construction of the attractor of an iterated function 
system. Indeed, instead of using the same family of maps at each construction
step, different families with different numbers of maps are 
allowed in a code tree.
Setting $I=\{1,\dots,M\}$, the length of a word $\tau\in I^k$ is $|\tau|=k$.
Consider a function $\omega\colon\bigcup_{k=0}^\infty I^k\to\Lambda$, where
$I^0=\{\emptyset\}$. We associate to $\omega$ a natural tree 
rooted at 
$\emptyset$ as follows: Let $\Sigma^{\omega}_{*}\subset\bigcup_{k=0}^\infty I^k$
be the unique set satisfying the following conditions:
\begin{itemize}
\item $\emptyset\in\Sigma^\omega_*$,
\item if $i_1\cdots i_k\in\Sigma^\omega_*$ and
   $\omega(i_1\cdots i_k)=\lambda$, then $i_1\cdots i_kl\in\Sigma^\omega_*$
   if and only if $l\leq M_\lambda$, 
\item if $i_1\cdots i_k\notin\Sigma^\omega_*$, then for all $l$ we have
   $i_1\cdots i_kl\notin\Sigma^\omega_*$.
\end{itemize}
The function $\omega$ restricted to $\Sigma^\omega_*$ is called
an {\it $\mathcal F$-valued code tree} and the set of all $\mathcal F$-valued 
code trees is denoted by $\Omega$. 
Note that in a code tree the vertex 
$i_1\cdots i_k$ may be identified with the function system 
$F^{\omega(i_1\cdots i_k)}$, and moreover, the edge connecting $i_1\cdots i_k$ to
$i_1\cdots i_kl$ may be identified with the map $f_l^{\omega(i_1\cdots i_k)}$.
{\it A sub code tree} of a code tree $\omega$ is the restriction of $\omega$ 
to a subset $B\subset\Sigma_*^\omega$, where $B$ is rooted at some vertex 
$i_1\cdots i_k\in\Sigma_*^\omega$ and $B$ contains all descendants of 
$i_1\cdots i_k$ which belong to $\Sigma_*^\omega$. We endow $\Omega$ with the 
topology generated by the sets 
\[
\{\omega\in\Omega\mid\Sigma_*^\omega\cap\bigcup_{j=0}^k I^j=J\text{ and }
 \omega(\bi)\in U_\bi\text{ for all }\bi\in J\},
\]
where $k\in\mathbb N$, $U_\bi\subset\Lambda$ is open for all $\bi\in J$ and 
$J\subset\bigcup_{j=0}^k I^j$ is a tree rooted at $\emptyset$ and having all 
leaves in $I^k$. With this topology functions $\omega_1$ and $\omega_2$ are 
``close'' to each other if their supports $\Sigma_*^{\omega_1}$ and 
$\Sigma_*^{\omega_2}$ agree up to the level $k$ and the values $\omega_1(\bi)$
and $\omega_2(\bi)$ are ``close'' to each other for all words $\bi$ with 
$|\bi|\le k$.  

We equip $I^{\mathbb N}$ with the product topology. For each code tree
$\omega\in\Omega$, define
\[
\Sigma^\omega=\{\bi=i_1i_2\cdots\in I^{\mathbb N}\mid
  i_1\cdots i_n\in\Sigma^\omega_*\text{ for all }n\in\mathbb N\}.
\]
Then $\Sigma^\omega$ is compact.
For all $k\in\mathbb N$ and $\bi\in\Sigma^\omega\cup\bigcup_{j=k}^\infty I^j$,
let $\bi_k=i_1\cdots i_k$ be the initial word of $\bi$ with length $k$.
We use the following type of natural abbreviations for compositions:
\[
f^\omega_{\bi_k}=f_{i_1}^{\omega(\emptyset)}\circ f_{i_2}^{\omega(i_1)}
  \circ\dotsb\circ f_{i_k}^{\omega(i_1\cdots i_{k-1})}\text{ and }
T_{\bi_k}^\omega=T_{i_1}^{\omega(\emptyset)}T_{i_2}^{\omega(i_1)}\cdots
  T_{i_k}^{\omega(i_1\cdots i_{k-1})}.
\]
Observe that, by the definition of the topology on $\Omega$, the maps 
$\omega\mapsto f^\omega_{\bi_k}$ and $\omega\mapsto T_{\bi_k}^\omega$ are Borel
measurable. The code tree fractal corresponding to $\omega\in\Omega$ is
$A^\omega =\{ Z^\omega(\bi)\mid \bi\in\Sigma^\omega\}$, where
$Z^{\omega}(\bi)=\lim_{k\to\infty}f^\omega_{\bi_k}(0)$. Note that the attractor
$A^\omega$ is well-defined since the maps $f_i^\lambda$ are uniformly
contracting and the translation vectors $a_i^\lambda$ belong to a bounded set.
For $k\in\mathbb N$, $\omega\in\Omega$ and $\bi\in\Sigma^\omega$, the 
{\it cylinder of length $k$ determined by $\bi$} is
\[
[\bi_k]=\{\bj\in\Sigma^\omega\mid j_l=i_l\text{ for all }l=1,\dots,k\}.
\]

Next we introduce the concept of a neck level which is 
an essential feature
of our model. The existence of neck levels guarantees that in our setting
the attractor is globally nearly homogeneous. In fact, if 
$N_m\in\mathbb N$ is a 
neck level of $\omega$, then all the sub code trees of $\omega$ rooted at
vertices $\bi\in\Sigma_*^\omega$ with $|\bi|=N_m$ are identical. In particular,
the attractor $A^\omega$ is a finite union of affine copies of the attractor of
the common sub code tree. Neck levels play an important role in the study of 
$V$-variable fractals, see for example \cite{BHS2005}, \cite{BHS2008} and 
\cite{BHS12}.   

{\it A neck list} $N=(N_m)_{m\in\mathbb N}$ is an increasing sequence of 
natural numbers. Let $\widetilde\Omega$ be the set of
$(\omega,N)\in\Omega\times\mathbb N^{\mathbb N}$ satisfying
\begin{itemize}
\item $N_m<N_{m+1}$ for all $m\in\mathbb N$ and
\item if $\bi_{N_m}\bj_l,\bi'_{N_m}\in\Sigma_*^\omega$, then
$\bi'_{N_m}\bj_l\in\Sigma_*^\omega$ and
$\omega(\bi_{N_m}\bj_l)=\omega(\bi'_{N_m}\bj_l)$.
\end{itemize}
The first condition means that $N$ is a neck list and the second 
condition 
guarantees that the sub code trees rooted at a certain neck level are 
identical.
{\it A shift } $\Xi\colon\widetilde\Omega\to\widetilde\Omega$ is defined by
means of 
neck levels, that is, $\Xi(\omega,N)=(\hat\omega,\hat N)$, where
$\hat N_m=N_{m+1}-N_1$ and $\hat\omega(\bj_l)=\omega(\bi_{N_1}\bj_l)$
for all $m,l\in\mathbb N$. We denote the elements of $\widetilde\Omega$
by $\tilde\omega$, 
and for all $i\in\mathbb N$ we write $N_i(\tilde\omega)=N_i$
for the projection of $\tilde\omega=(\omega,N)$ onto the $i$-th coordinate of 
$N$. Moreover, on $\widetilde\Omega$ we use the topology
generated by the {\it cylinders}
\begin{align*}
[(\omega,N)_m]=\{(\hat\omega,\hat N)\in\widetilde\Omega &\mid
  \hat N_i=N_i\text{ for all }i\le m\text{ and }\hat\omega(\tau)=\omega(\tau)\\
&\text{ for all }\tau\text{ with } |\tau|<N_m\}.
\end{align*}
For any function $\phi$ of $\omega$ we use the notation $\phi(\tilde\omega)$ to
view $\phi$ as a function of $\tilde\omega$. Finally for 
all $n<m\in\mathbb N\cup\{0\}$, let
\[
\Sigma_*^{\tilde\omega}(n,m)=\{i_{N_n+1}\cdots i_{N_m}\mid\bi_{N_n}i_{N_n+1}\cdots
  i_{N_m}\in\Sigma_*^{\tilde\omega}\},
\]
where $N_0=0$.

For the purpose of defining the pressure, we proceed by recalling the notation
from \cite{F88}. Let $T:\mathbb R^d\to\mathbb R^d$ be a non-singular linear 
mapping and let
\[
0<\sigma_d\leq\sigma_{d-1}\leq\dots\leq\sigma_2\leq\sigma_1=\Vert T\Vert
\]
be the singular values of $T$, that is, the lengths of the semi-axes of
the ellipsoid $T(B(0,1))$, where $B(x,\rho)\subset\mathbb R^d$ is the closed
ball with radius $\rho>0$ centred at $x\in\mathbb R^d$.
We define the {\it singular value function} by
\[
\Phi^s(T)
 =\begin{cases}\sigma_1\sigma_2\cdots\sigma_{m-1}\sigma_m^{s-m+1},&
                \text{if } 0\le s\le d,\\
   \sigma_1\sigma_2\cdots\sigma_{d-1}\sigma_d^{s-d+1},&\text{if }s>d,
  \end{cases}
\]
where $m$ is the integer such that $m-1\le s<m$.  The singular value function
is submultiplicative, that is,
\[
\Phi^s(TU)\le\Phi^s(T)\Phi^s(U)
\]
for all linear maps $T,U:\mathbb R^d\to\mathbb R^d$. For further properties of
the singular value function see for example \cite{F88}. We assume that there
exist $\underline\sigma,\overline\sigma\in (0,1)$ such that
\[
0<\underline\sigma\le\sigma_d(T_i^\lambda)\le\sigma_1(T_i^\lambda)
  \le\overline\sigma<1
\]
for all $\lambda\in\Lambda$ and for all $i=1,\dots,M_\lambda$. Note that,
whilst the condition $\overline\sigma<1$ follows from the uniform
contractivity assumption, the existence of $\underline\sigma>0$ is an
additional assumption.

For all $k\in\mathbb N$ and $s\ge 0$,
let
\[
S^{\tilde\omega}(k,s)=\sum_{\bi_k\in\Sigma_*^{\tilde\omega}}
  \Phi^s(T_{\bi_k}^{\tilde\omega}).
\]
The {\it pressure} is defined as follows
\begin{equation}\label{pressure}
p^{\tilde\omega}(s)=\lim_{k\to\infty}\frac{\log S^{\tilde\omega}(k,s)}k
\end{equation}
provided that the limit exists. Since $T\mapsto\Phi^s(T)$ is 
a continuous 
function, the map $\tilde\omega\mapsto p^{\tilde\omega}(s)$ is Borel measurable.

According to the following theorem, the pressure exists and has a unique zero
for typical random affine code tree fractals.

\begin{theorem}\label{pexists}
Assume that $P$ is an ergodic $\Xi$-invariant Borel 
probability measure on
$\widetilde\Omega$ such that
$\int_{\widetilde\Omega}N_1(\tilde\omega)\,dP(\tilde\omega)<\infty$.
Then for $P$-almost all $\tilde\omega\in\widetilde\Omega$ the pressure
$p^{\tilde\omega}(s)$ exists for all $s\in[0,\infty[$.
Furthermore, $p^{\tilde\omega}$ is strictly decreasing and there
exists a unique $s_0$ such that $p^{\tilde\omega}(s_0)=0$ for
$P$-almost all $\tilde\omega\in\widetilde\Omega$.
\end{theorem}

\begin{proof}
See \cite[Theorem 4.3]{JJKKSS}.
\end{proof}

In \cite[Remark 2.1]{JJKKSS} it was shown that any compact subset 
of the attractor 
of an iterated function system is a code tree fractal and, in particular, any 
sub-self-affine set is a code tree fractal. While verifying this, one ends up 
studying subsystems of the original iterated function system. 
For example, suppose that $F^1=\{f_1,f_2,f_3\}$ and let
$F^2=\{f_1,f_2\}$ and $F^3=\{f_2,f_3\}$. When changing the 
translation vector of the second map in $F^2$, one needs to modify also the 
translation vector of the first map in $F^3$ since these maps are the same.
Therefore, it is useful to allow identifications of translation vectors 
between different families. For this purpose, we 
equip the set
$\widehat\Lambda=\{(\lambda,i)\mid\lambda\in\Lambda, i=1,\dots,M_\lambda\}$
with an equivalence relation $\sim$ satisfying the following assumptions
\begin{itemize}
\item the cardinality $\mathcal A$ of the set of equivalence classes
      $\ba:=\widehat\Lambda/\sim$ is finite,
\item for every $\lambda\in\Lambda$ we have $(\lambda,i)\sim (\lambda,j)$ if
      and only if $i=j$ and
\item the equivalence classes, regarded as subsets of $\Lambda$, 
are Borel sets.
\end{itemize}
The notation $\ba$ for the set of equivalence classes refers to the fact
that some translation vectors of the maps $f_i^\lambda$ are identified even
though the maps are not. The second condition means that different 
translation
vectors inside a system $F^\lambda$ are never identified. The first condition
allows us to view the set of equivalence classes $\ba$ as an element
of $\mathbb R^{d\mathcal A}$. From now on we will write $A_\ba^{\tilde\omega}$ for
the attractor of a code tree $\tilde\omega$ to emphasise that it depends on
the set of equivalence classes of translation vectors $\ba$.

Now we are ready to state our main theorem in this section.
Generalising the earlier results in \cite{JJKKSS}, we prove that, under
the assumptions of Theorem~\ref{pexists}, for random affine code
tree fractals the Hausdorff, packing and box counting dimensions, denoted
by $\dimH$, $\dimp$ and $\dimb$, respectively, are almost surely
equal to the unique zero of the pressure provided that a probabilistic
version of the F-S condition is satisfied. We denote by $s_0$ the
unique zero of the pressure given by Theorem \ref{pexists}.

\begin{theorem}\label{maintheorem}
Assume that $0<\underline\sigma\le\overline\sigma<\frac 12$.
Let $P$ be an ergodic $\Xi$-invariant Borel
probability measure on $\widetilde\Omega$ such that
$\int_{\widetilde\Omega}N_1(\tilde\omega)\,dP(\tilde\omega)<\infty$.
Suppose that for all $0<s<d$
\begin{equation}\label{probcondcs}
P\{\tilde{\omega}\in\widetilde\Omega\mid\{T_\bj^{\tilde{\omega}}\mid\bj
  =\bi_l, 1\le l\le N_1\text{ and }\bi_{N_1}\in\Sigma_*^{\tilde\omega}(0,1)\}
  \text{ satisfies condition }C(s)\}>0.
\end{equation}
Then for $P$-almost all $\tilde\omega\in\widetilde\Omega$,
\[
\dimH(A_\ba^{\tilde\omega})=\dimp(A_\ba^{\tilde\omega})=\dimb(A_\ba^{\tilde\omega})
   =\min\{s_0,d\}
\]
for $\mathcal{L}^{d\mathcal{A}}$-almost all $\ba\in\mathbb R^{d\mathcal A}$.
\end{theorem}

\begin{remark}\label{oldthm}
a) In \cite[Theorem 5.1]{JJKKSS} a special case of Theorem~\ref{maintheorem}
was proven under substantially stronger assumptions. First of all,
\cite[Theorem 5.1]{JJKKSS} deals only with the planar case $d=2$.
Moreover, instead of \eqref{probcondcs} the following non-existence of 
parallelly mapped vectors is assumed 
\begin{equation}\label{parallel}
\begin{split}
P\{\tilde{\omega}\in\widetilde\Omega\mid&\text{ there exists }v\in\mathbb R^2
    \setminus\{0\}\text{ such that }T_{\bi_{N_1}}^{\tilde{\omega}}(v)
  \text{ are parallel}\\
  &\text{ for all }\bi_{N_1}\in\Sigma_*^{\tilde\omega}(0,1)\}<1.
\end{split}
\end{equation}
Observe that in the case $d=2$ the condition $C(s)$ is equivalent to the
condition $C(1)$ for all $0<s<2$. Furthermore, for a family $\{S_i\}_{i=1}^k$
condition $C(1)$ means that for all vectors
$v,w\in\mathbb R^2\setminus\{0\}$ there exists $i$ such that
$\langle S_iv\mid w\rangle\ne 0$. Therefore, condition \eqref{parallel} implies
condition \eqref{probcondcs} in the case $d=2$. Condition \eqref{probcondcs}
is weaker than condition \eqref{parallel}, since in the former one all iterates
up to level $N_1$ are considered whilst in the second one only iterates at
level $N_1$ play a role. In \cite[Theorem 5.1]{JJKKSS} there
are also technical conditions concerning the measure $P$ which are not
needed here. As explained in \cite{JJKKSS} the upper bound $\frac 12$ for
$\overline\sigma$ is optimal in Theorem~\ref{maintheorem}.

b) The map $N_1(\tilde\omega)$ is Borel measurable 
as a projection. 
Since 
$\tilde\omega\mapsto T_\bj^{\tilde\omega}$ is a Borel map for all finite words
$\bj$ and the set of families of linear maps satisfying condition $C(s)$ is
open, the set in \eqref{probcondcs} is a Borel set.
\end{remark}

Before the proof of Theorem~\ref{maintheorem} we present an example which
demonstrates how certain random $V$-variable and random graph directed
systems fit in our framework.    

\begin{example}\label{graphexample}
Let $\Lambda$ be a finite set of directed labeled multigraphs
$\lambda=(W, E^\lambda,\mathcal F^\lambda)$ where
$W=\{1,2,...,V\}$ is the common finite set of vertices for all
$\lambda\in\Lambda$, $E^\lambda$ is a finite set of directed edges
and, for each directed edge $e\in E^\lambda$, there is an
associated map $\phi_e^\lambda\in\mathcal F^\lambda $ which is a contraction on
$\mathbb R^d$. For all edges $e$, we denote by $i(e)$ and $t(e)$ 
the initial and terminal vertices of $e$, respectively. 

Recall that in the general setting of
graph directed systems (see for example \cite{MU03}), for each vertex $v\in W$,
there is an associated metric space $X_v$, and for each edge $e\in E^\lambda$,
the associated map is $\phi_e^\lambda:X_{t(e)}\to X_{i(e)}$. Here we 
make the simplifying assumption that $X_v=\mathbb R^d$ for all $v\in V$. Let
\[
M=\max_{\substack{v\in W\\\lambda\in\Lambda}}\#\{e\in E^\lambda\mid i(e)=v\}
\]
be the maximum
number of maps within any fixed graph $\lambda\in\Lambda$ with the same range.
Recall that in a deterministic graph directed system there is only one graph
$\lambda$ and the composition $\phi_{e_1}\circ\phi_{e_2}$ is allowed provided that
$t(e_1)=i(e_2)$. In some random graph directed models (see for example 
\cite{RU}) the graph $\lambda$ is fixed and the maps $\phi_e$ are random 
whereas in our model the graphs are allowed to be random as well.  

Fix a probability measure  $\mu$ on $\Lambda$ and set 
$\mathcal G=\Lambda^{\{0\}\cup\mathbb N}$. Let 
$\mu^\infty=\mu^{\{0\}\cup\mathbb N}$ be the product measure on $\mathcal G$ and 
let $\sigma:\mathcal G\to\mathcal G$,
\[
\sigma(g_0g_1\cdots)=g_1g_2\cdots\text{ for all }\bg=g_0g_1\cdots\in\mathcal G,
\]
be the left shift. To all $\bg\in\mathcal G$, we associate
a $V$-tuple of code trees $\omega=(\omega_1,...,\omega_V) $ as follows:
For all $\lambda\in\Lambda$ and $v\in \{1,2,...,V\}$, let
$\mathcal F_v^\lambda=\{\phi_e^\lambda\mid  e\in E^\lambda\text{ and }i(e)=v\}$
be the iterated function system consisting of those maps in $\lambda$ whose
ranges correspond to the vertex $v$. We write $I=\{1,\dots,M\}$ and rename the 
edges with $i(e)=v$ as $e_1,\dots,e_m$. Observe that $m$ may 
depend on $v\in W$ and $\lambda\in\Lambda$. The definition of $M$ implies that
$m\le M$. For all $v\in W$, set
$\omega_v(\emptyset)=\mathcal F_v^{g_0}$. Now we proceed inductively. 
Assuming that
$\omega_v(i_1\cdots i_n)=\mathcal F_w^{g_n}
  =\{\phi_{e_1}^{g_n},\ldots,\phi_{e_m}^{g_n}\}$ 
for some $w\in W$, define 
$\omega_v(i_1\cdots i_ni_{n+1})=\mathcal F_{t(e_{i_{n+1}})}^{g_{n+1}}$ for
$i_{n+1}=1,\dots,m$. Observe that every $\bg\in\mathcal G$ defines a sequence of
graphs, which, in turn, determines a sequence of ordered walks
starting from $v$. The code tree fractal corresponding to $\omega_v$ is the 
set of the limit points of the set of maps associated to all infinite paths  
starting from $v$. This code tree fractal is the $v$-th component in the graph 
directed set corresponding to the infinite sequence $\bg$.

A $V$-tuple $\omega$ of code trees defines a $V$-tuple of code tree 
fractals $\bar A^\omega=(A_1^\omega,\dots,A_V^\omega)$ componentwise as 
described at the beginning of this section. Note that for fixed 
$\bg\in\mathcal G$, any sub code tree rooted at level $n$ is determined by the 
code of its top node. Since this code is an element of the set 
$\{\mathcal F_k^{g_n}\}_{k=1}^V$, there are at most $V$ distinct code trees at 
a fixed level. By definition, this means that 
$\omega=(\omega_1,\dots,\omega_V)$ and the corresponding code tree fractals, 
$\{A_v^\omega\mid v\in W\}$, are $V$-variable.

In order to apply Theorem~\ref{maintheorem} to the above system, we need
some further assumptions. Suppose that 
$\phi_e^\lambda(x)=T^\lambda_e(x)+a_e^\lambda$ is a non-singular affine map on 
$\mathbb R^d$ with singular values uniformly bounded from below by 
$\underline\sigma>0$ and from above by $\overline\sigma <\frac 12$ for all
$\lambda\in\Lambda$ and $e\in E^\lambda$. We equip the set 
$\widehat\Lambda=\{(\lambda,e)\mid\lambda\in\Lambda\text{ and }e\in E^\lambda\}$
with the trivial equivalence relation, that is,  
$(\lambda,e)\sim(\lambda',e')$ if $(\lambda,e)=(\lambda',e')$. Then the
set of equivalence classes $\ba=\widehat\Lambda/\sim$ may be
identified with the collection of all translation vectors. 
Since $\Lambda$ is finite and the number of 
edges is bounded, the number $\mathcal A$ of equivalence classes 
in $\ba$ is finite, and therefore, $\ba\in\mathbb R^{d\mathcal A}$. 
To ensure that the $V$-tuple of code trees corresponding to 
$\bg$ has no ``dying'' branches and, in particular, defines a 
$V$-tuple of non-empty code tree fractals,
we assume that in $\mu$-almost all graphs $\lambda\in\Lambda$ every vertex 
is an initial vertex of some edge, that is,  
\[
\mu\{\lambda\in\Lambda\mid\text{ for all }v\in W\text{ there exists }
  e\in E^\lambda\text{ with }i(e)=v\}=1.
\]

In addition to the above assumptions, the existence of neck levels needs to
be guaranteed. Recall that at a neck level all the sub code trees are 
identical. Such levels exist provided that there is a vertex $v_0\in W$ such 
that the $\mu$-measure of the set of graphs $\lambda\in\Lambda$ 
whose all edges have terminal vertex equal to $v_0$ is positive. Hence, 
we assume that there exists a vertex $v_0\in W$ such that
$\mu(\Lambda_{\text{neck}})>0$ where 
\[
\Lambda_{\text{neck}}=\{\lambda\in\Lambda\mid t(e)=v_0\text{ for all }
  e\in E^\lambda\}.
\]
We emphasise that this is a natural assumption for a collection of random 
graphs.  For example, it is satisfied if the random graphs are 
constructed as follows: First choose for each $v\in W$ the number of edges with 
initial vertex equal to $v$. Then for each edge choose the terminal vertex 
independently according to a probability vector $(p_1,\dots,p_V)$ with 
$p_{v_0}>0$.  We first define auxiliary neck levels inductively as follows: Set
\[
\tilde N_1(\bg)=\min\{n\ge 0\mid t(e)=v_0\text{ for all }e\in E^{g_n}\}+1 
\]
and define  
\[
\tilde N_{k+1}(\bg)=\min\{n\ge \tilde N_k(\bg)\mid t(e)=v_0\text{ for all }e\in 
 E^{g_n}\}+1. 
\]
This sequence is well defined for $\mu^\infty$-almost all $\bg\in\mathcal G$ 
since the distances $\tilde N_{k+1}-\tilde N_k$
form a sequence of independent geometrically distributed random variables, and 
therefore, for the expectation we have 
\begin{equation}\label{Ebounded}
\int\tilde N_k(\bg)\,d\mu^\infty(\bg)
  =k\int\tilde N_1(\bg)\,d\mu^\infty(\bg)<\infty
\end{equation}
for all $k\in\mathbb N$. The neck list is defined by $N_k=\tilde N_{2n_0^2k}$ 
for all $k\in\mathbb N$, where $n_0$ is as in Theorem~\ref{CsCm}.

Observe that the existence of neck levels implies that $A_{v_0}^\omega$ is 
a finite union of affine
copies of the attractor determined by the common sub code tree at the first
neck level $N_1$. Since all the sub code trees at this level are identical,
all the components of the $V$-tuble attractor $\bar A^\omega$ are
finite unions of affine images of the same fixed set. Thus the dimensions 
of the components of $\bar A^\omega$ are equal to that of $A_{v_0}^\omega$. 
For the purpose of calculating the almost sure dimension value of
$A_{v_0}^\omega$, we apply Theorem~\ref{maintheorem}. 

We proceed by verifying that the assumptions of Theorem~\ref{maintheorem} are
satisfied. Since we
attached to almost every code tree $\omega_{v_0}$ a unique neck list, we 
may identify $\widetilde\Omega$ with the space of all code trees $\omega_{v_0}$.
Moreover, the product measure $\mu^\infty$ determines a mixing, thereby ergodic,
$\Xi$-invariant measure $P$ on $\widetilde\Omega$. Now \eqref{Ebounded} and the
definition of $N_1$ imply that 
$\int N_1(\tilde\omega_{v_0})\,d P(\tilde\omega_{v_0})<\infty$.  

Finally, we have to ensure that the F-S condition \eqref{probcondcs} is valid.
Intuitively, this is achieved if we assume that there are many 
allowed sequences of edges with 
initial and terminal vertices equal to $v_0$ such that the associated maps 
satisfy the assumptions of Theorem~\ref{CsCm}. More precisely, we suppose
that there exists $l\in\mathbb N$ such that
\begin{equation}\label{nicemaps}
\begin{split}
\mu^l\{&(\lambda_1,\dots,\lambda_l)\in\Lambda^l\mid\lambda_j\not\in
  \Lambda_{\text{preneck}}\text{ for }j=1,\dots,l-1,\,\lambda_l\in
  \Lambda_{\text{preneck}},\text{ there exist }\\
 &e_{i_1}^{\lambda_1}\cdots e_{i_l}^{\lambda_l}\text{ and }e_{j_1}^{\lambda_1}\cdots 
  e_{j_l}^{\lambda_l}\text{ with }i(e_{i_1}^{\lambda_1})=i(e_{j_1}^{\lambda_1})
  =t(e_{i_l}^{\lambda_l})=t(e_{j_l}^{\lambda_l})=v_0\text{ and}\\
 &F:=T_{e_{i_1}}^{\lambda_1}\cdots T_{e_{i_l}}^{\lambda_l}\text{ and }
  G:=T_{e_{j_1}}^{\lambda_1}\cdots T_{e_{j_l}}^{\lambda_l}
  \text{ satisfy the assumptions of Theorem }\ref{CsCm}\}\\
 &>0. 
\end{split}
\end{equation}  
Since we use the product measure $\mu^\infty$ on $\mathcal G$, 
there is positive
probability that the same pair of maps $(F,G)$ appears successively $2n_0^2$ 
times. Therefore, from Theorem~\ref{CsCm} we see that the condition 
\eqref{probcondcs} is satisfied. 
Observe that the condition \eqref{nicemaps} is satisfied with $l=1$ if there 
are maps $\phi_e^\lambda$ and $\phi_{e'}^\lambda$ as in Theorem~\ref{CsCm} with
$i(e)=i(e')=t(e)=t(e')=v_0$ and $\lambda\in\Lambda_{\text{preneck}}$ is chosen
with positive probability. This, in 
turn, is true for typical families by Corollary~\ref{typical}. 
\end{example}

For the proof of Theorem~\ref{maintheorem} we need the following notation and
auxiliary results.

\begin{definition}\label{full}
Let $c>0$ and $0<s<d$. We say that a family of non-singular linear mappings
$\{S_j:\mathbb R^d\to\mathbb R^d\}_{j=1}^k$ is {\it $(c,s)$-full} if
\[
\sum_{j=1}^k\Phi^s(US_jV)\ge c\Phi^s(U)\Phi^s(V)
\]
for all non-singular linear mappings $U,V:\mathbb R^d\to\mathbb R^d$.
\end{definition}

In Lemmas \ref{fullok} and \ref{numbersatcondition} we explore consequences
of the probabilistic version of the F-S condition \eqref{probcondcs}.

\begin{lemma}\label{fullok}
Assuming that the condition \eqref{probcondcs} is satisfied, there exists $c>0$
such that
\[
\varrho=P\{\tilde\omega\in\widetilde\Omega|\{T_{\bi_{N_1}}^{\tilde\omega}\}_
  {\bi_{N_1}\in\Sigma_*^{\tilde\omega}(0,1)}\text{ is }(c,s)\text{-full }\}>0.
\]
\end{lemma}

\begin{proof}
Since the set of $(c,s)$-full families is a Borel set, the set in 
the definition of $\varrho$ is a Borel set.
Let $U,V:\mathbb R^d\to\mathbb R^d$ be non-singular linear maps. 
Suppose that
\[
\mathcal F=\{T_\bj^{\tilde{\omega}}\mid\bj=\bi_l, 
1\le l\le N_1\text{ and }\bi_{N_1}\in\Sigma_*^{\tilde\omega}(0,1)\}
\]
satisfies the condition $C(s)$. 
By the proof of \cite[Proposition 2.1]{FS} (see also
\cite[Corollary 2.2]{FS}), there exists $\bj$ such that
\begin{equation}\label{basicineq}
\Phi^s(UT_\bj^{\tilde\omega}V)\ge C(\mathcal F)\Phi^s(U)
  \Phi^s(V),
\end{equation}
where the constant $C(\mathcal F)$ is independent of $U$ and $V$. 
Observe that $C(\mathcal F)$ depends on $s$ but it is an interpolation of 
the constants 
obtained by replacing $s$ by $m$ and $m+1$, where $m$ is the integer part of 
$s$ (recall Remark~\ref{FSremark}). Let 
$\overline\bi_{N_1}\in\Sigma_*^\omega(0,1)$ be such that 
$\bj=\overline\bi_{|\bj|}$.
Writing 
$T_{\overline\bi_{N_1}}^{\tilde\omega}
=T_{\bj}^{\tilde\omega}T_{i_{|\bj|+1}}^{\tilde\omega(i_{\vert\bj\vert})}
\cdots T_{i_{N_1}}^{\tilde\omega(i_{N_1-1})}$
and applying \eqref{basicineq}, gives
\[
\Phi^s(UT_{\overline\bi_{N_1}}^{\tilde\omega}V)
\ge\underline\sigma^{N_1-\vert\bj\vert}\Phi^s(UT_{\bj}V)
\ge C(\mathcal F) \underline\sigma^{N_1}\Phi^s(U)\Phi^s(V).
\]
This implies that
\begin{equation}\label{FSCor}
\sum_{\bi_{N_1}\in\Sigma_*^{\tilde\omega}(0,1)}\Phi^s(UT_{\bi_{N_1}}^{\tilde\omega}V)
  \ge C(\mathcal F)\underline\sigma^{N_1}\Phi^s(U)
   \Phi^s(V)
\end{equation}
for all linear mappings $U,V:\mathbb R^d\to\mathbb R^d$. From \eqref{probcondcs}
we conclude that there exists $c>0$ such that
\[
P\{\tilde{\omega}\in\widetilde\Omega\mid C(\mathcal F)
  \underline\sigma^{N_1}>c\}>0,
\]
giving the claim.
\end{proof}

In the following lemma we denote by $\# A$ the number of elements in 
a set $A$.

\begin{lemma}\label{numbersatcondition}
Assume that the condition \eqref{probcondcs} is satisfied and let 
$\varrho$ and $c$ be as in Lemma~\ref{fullok}. Define for all
$n,m\in\mathbb{N}$ 
\[
E^{\tilde\omega}(n,n+m)=\#\{n<j\le n+m\mid\{T_{\bi_{N_1}}^{\Xi^{j-1}
  (\tilde\omega)}\}\text{ is }(c,s)\text{-full }\}
\]
and suppose that $P$ is $\Xi$-invariant and ergodic. Then for $P$-almost all
$\tilde\omega\in\widetilde\Omega$ the following is true: 
for all $\varepsilon>0$ there exists
$n_1(\tilde\omega,\varepsilon)>0$ such that for all
$n>n_1(\tilde\omega,\varepsilon)$ we have
\[
E^{\tilde\omega}(n,n+\lceil\varepsilon n\rceil)\ge 1,
\]
where $\lceil x\rceil$ is the smallest integer $m$ with $x\le m$.
\end{lemma}

\begin{proof}
Let $\chi$ be the characteristic function of the set
$\{\tilde\omega\in\widetilde\Omega\mid\{T_{\bi_{N_1}}^{\tilde\omega}\}\text{ is }
  (c,s)\text{-full}\,\}$.
Since
\[
E^{\tilde\omega}(0,n)=\sum_{j=0}^{n-1}\chi(\Xi^j(\tilde{\omega})),
\]
we obtain from the Birkhoff ergodic theorem that for $P$-almost all
$\tilde\omega\in\widetilde\Omega$ 
\begin{equation}\label{expectedfull}
\lim_{n\to\infty}\frac{E^{\tilde\omega}(0,n)}{n}=\int_{\tilde\Omega}
  \chi(\tilde\omega)dP(\tilde\omega)=\varrho.
\end{equation}

Fix $\tilde\omega\in\widetilde\Omega$ satisfying \eqref{expectedfull} and let
$\varepsilon>0$. Defining
$0<\tilde{\varepsilon}=\frac{\varrho\varepsilon n-1}{(\varepsilon+2)n}<\varrho$
for sufficiently large $n$,
there exists $n_1(\tilde{\omega},\varepsilon)>0$ such that for all
$n>n_1(\tilde{\omega},\varepsilon)$ and for all $m\ge 0$ we have
\[
(\varrho-\tilde\varepsilon)(n+m)<E^{\tilde\omega}(0,n+m)<(\varrho
   +\tilde\varepsilon)(n+m),
\]
and therefore,
\[
E^{\tilde\omega}(n,n+m)=E^{\tilde\omega}(0,n+m)-E^{\tilde\omega}(0,n)
  >(\varrho-\tilde\varepsilon)m-2\tilde\varepsilon n.
\]
Finally, taking $m\ge \varepsilon n$, gives
$(\varrho-\tilde\varepsilon)m-2\tilde\varepsilon n\ge 1$, which implies that
$E^{\tilde\omega}(n,n+m)\ge 1$. In particular,
$E^{\tilde\omega}(n,n+\lceil\varepsilon n\rceil)\ge 1$.
\end{proof}

\begin{lemma}\label{numbernecks}
Under the assumptions of Theorem~\ref{pexists}, we have for $P$-almost all 
$\tilde\omega\in\widetilde\Omega$ that
\[
\lim_{n\to\infty}\frac{N_{n+\lceil\varepsilon n\rceil}(\tilde\omega)
  -N_{n-1}(\tilde\omega)}{N_n(\tilde\omega)}=\varepsilon
\]
for all $\varepsilon>0$.
\end{lemma}

\begin{proof}
Since $N_n({\tilde\omega})=\sum_{j=0}^{n-1}N_1(\Xi^j(\tilde\omega))$,
the Birkhoff ergodic theorem implies that for $P$-almost all
$\tilde\omega\in\widetilde\Omega$ 
\[
\lim_{n\to\infty}\frac{N_n({\tilde\omega})}{n}=\int_{\tilde\Omega}
  N_1(\tilde\omega)dP(\tilde\omega)=b<\infty.
\]
Now for any typical $\tilde\omega$ we have
\[
\lim_{n\to\infty}\frac{N_{n+\lceil\varepsilon n\rceil}(\tilde\omega)}
   {N_n(\tilde\omega)}
  =\lim_{n\to\infty}\frac{N_{n+\lceil\varepsilon n\rceil}(\tilde\omega)}
  {n+\lceil\varepsilon n\rceil}\cdot\frac{n+\lceil\varepsilon n\rceil}{n}\cdot
   \frac{n}{N_n(\tilde\omega)}=b(1+\varepsilon)\frac 1b=1+\varepsilon,
\]
and similarly we see that
$\lim_{n\to\infty}\frac{N_{n-1}(\tilde\omega)}{N_n(\tilde\omega)}=1$.
Therefore,
\[
\lim_{n\to\infty}\frac{N_{n+\lceil\varepsilon n\rceil}(\tilde\omega)
  -N_{n-1}(\tilde\omega)}{N_n(\tilde\omega)}=\varepsilon.
\]
\end{proof}

Now we are ready to prove Theorem~\ref{maintheorem}.

\begin{proof}[Proof of Theorem~\ref{maintheorem}]
In \cite[(5.20)]{JJKKSS} it is proven that under the assumptions of
Theorem~\ref{pexists} we have $\udimb(A_\ba^{\tilde\omega})\le\min\{s_0,d\}$
for $P$-almost all $\tilde\omega\in\widetilde\Omega$. Here $\udimb$
is the upper box counting dimension. Note that the 
assumption $d=2$ is not needed in the proof of \cite[(5.20)]{JJKKSS}.
Since always $\dimH\le\dimp\le\udimb$ (see for example 
\cite[(3.17) and (3.29)]{F}), it is sufficient to verify that 
\begin{equation}\label{almostgoal}
\dimH(A_\ba^{\tilde\omega})\ge\min\{s_0,d\}
\end{equation}
for $P$-almost all $\tilde\omega\in\widetilde\Omega$. Let $s<\min\{s_0,d\}$. 
In the proof of \cite[Theorem 3.2]{JJKKSS} it is shown that \eqref{almostgoal} 
follows provided 
that for $P$-almost all $\tilde\omega\in\widetilde\Omega$ there exists a
probability measure $\mu^{\tilde\omega}$ on $\Sigma^{\tilde\omega}$ and a constant
$D(\tilde{\omega})>0$ such that
\begin{equation}\label{eq66}
\mu^{\tilde\omega}([\bi_l])\le D(\tilde\omega)\Phi^s(T_{\bi_l}^{\tilde\omega})
\end{equation}
for all $\bi\in\Sigma^{\tilde\omega}$ and $l\in\mathbb N$.

For the purpose of verifying \eqref{eq66}, we define for all 
$\tilde\omega\in\widetilde\Omega$ and $m\in\mathbb N$
\begin{equation}\label{muomega}
\mu_m^{\tilde\omega} = \frac{\sum_{\bi_{N_m}\in\Sigma_*^{\tilde\omega}(0,m)}
   \Phi^s(T_{\bi_{N_m}}^{\tilde\omega})\delta_{\bi_{N_m}}}{\sum_{\bi_{N_m}
   \in\Sigma_*^{\tilde\omega}(0,m)}
   \Phi^s(T_{\bi_{N_m}}^{\tilde\omega})},
\end{equation}
where $\delta_{\bi_{N_m}}$ is the Dirac measure at some fixed point of the 
cylinder $[\bi_{N_m}]$. The choice of the cylinder point plays no role in 
what follows. Since $\Sigma^{\tilde\omega}$ is compact, the sequence
$(\mu_m^{\tilde\omega})_{m\in\mathbb N}$ has a weak*-converging
subsequence with a limit measure $\mu^{\tilde\omega}$. We proceed 
by showing that $\mu^{\tilde\omega}$ satisfies \eqref{eq66}.

By Lemma~\ref{numbernecks} the following is true for $P$-almost all 
$\tilde\omega\in\widetilde\Omega$: for all $\varepsilon>0$
there exists $n_2(\tilde\omega,\varepsilon)>0$ such
that for all $n>n_2(\tilde\omega,\varepsilon)$
\begin{equation}\label{neckbound}
N_{n+\lceil\varepsilon n\rceil}(\tilde\omega)-N_{n-1}(\tilde{\omega})
  <2\varepsilon N_n(\tilde\omega).
\end{equation}
Furthermore, it follows from the definition of the pressure that for 
$P$-almost all
$\tilde\omega\in\widetilde\Omega$ there exists for all $\varepsilon>0$ a
number $n_3(\tilde\omega,\varepsilon)>0$ such that for all
$n>n_3(\tilde\omega,\varepsilon)$ we have
\begin{equation}\label{pressurebound}
e^{(p^{\tilde\omega}(s)-\varepsilon)N_n(\tilde\omega)}<\sum_{\bi_{N_n}
   \in\Sigma_*^{\tilde\omega}(0,n)}\Phi^s(T_{\bi_{N_n}}^{\tilde\omega})
   < e^{(p^{\tilde\omega}(s)+\varepsilon)N_n(\tilde\omega)}.
\end{equation}

Let $\varepsilon>0$. Consider $\tilde\omega\in\widetilde\Omega$ satisfying
Lemma~\ref{numbersatcondition}, \eqref{neckbound} and
\eqref{pressurebound} and set
$n_0(\tilde\omega,\varepsilon)=\max\{n_1(\tilde\omega,\varepsilon),
  n_2(\tilde\omega,\varepsilon),n_3(\tilde\omega,\varepsilon)\}$.
For all $\bi_l\in\Sigma_*^{\tilde\omega}$ with
$l>N_{n_0(\tilde\omega,\varepsilon)}$, there exists
$n>n_0(\tilde\omega,\varepsilon)$ such that $N_{n-1}<l\leq N_n$. Now
Lemma~\ref{numbersatcondition} implies the existence of
$1\le k\le\lceil\varepsilon n\rceil$ such that
$\{T_{\bj_{N_1}}^{\Xi^{n+k-1}(\tilde{\omega})}\}$ is $(c,s)$-full. Let
$m$ be a natural number with $m>\varepsilon n$. In the remaining part of
the proof we use the following abbreviations 
$\sum_\bj=\sum_{\bj:\bi_l\bj\in\Sigma_*^{\tilde\omega}(0,n+k-1)}$,
$\sum_{N_1}=\sum_{\bj_{N_1}\in\Sigma_*^{\tilde\omega}(n+k-1,n+k)}$,
$\sum_{N_{m-k}}=\sum_{\bk_{N_{m-k}}\in\Sigma_*^{\tilde\omega}(n+k,n+m)}$ and
$\sum_{N_{n+k-1}}=\sum_{\bi_{N_{n+k-1}}\in\Sigma_*^{\tilde\omega}(0,n+k-1)}$, and
denote by $T_{(\bi_l)\bj}^{\tilde\omega}$ the last $|\bj|$ maps of
$T_{\bi_l\bj}^{\tilde\omega}$. Using the definition of
$\mu_{n+m}^{\tilde\omega}$, applying the submultiplicativity of $\Phi^s$ in
the numerator and utilising the $(c,s)$-fullness in the denominator, we
obtain
\begin{align*}
\mu_{n+m}^{\tilde\omega}([\bi_l])&=\frac{\sum_\bj\sum_{N_1}\sum_{N_{m-k}}
   \Phi^s(T_{\bi_l\bj}^{\tilde\omega}T_{\bj_{N_1}}^{\Xi^{n+k-1}(\tilde\omega)}
   T_{\bk_{N_{m-k}}}^{\Xi^{n+k}(\tilde\omega)})}
 {\sum_{N_{n+k-1}}\sum_{N_1}\sum_{N_{m-k}}\Phi^s(T_{\bi_{N_{n+k-1}}}^{\tilde\omega}
   T_{\bj_{N_1}}^{\Xi^{n+k-1}(\tilde\omega)}
   T_{\bk_{N_{m-k}}}^{\Xi^{n+k}(\tilde\omega)})}\\
&\le\frac{\Phi^s(T_{\bi_l}^{\tilde\omega})\sum_\bj\sum_{N_1}\sum_{N_{m-k}}
   \Phi^s(T_{(\bi_l)\bj}^{\tilde\omega})
   \Phi^s(T_{\bj_{N_1}}^{\Xi^{n+k-1}(\tilde\omega)})
   \Phi^s(T_{\bk_{N_{m-k}}}^{\Xi^{n+k}(\tilde\omega)})}
 {c\sum_{N_{n+k-1}}\sum_{N_{m-k}}\Phi^s(T_{\bi_{N_{n+k-1}}}^{\tilde\omega})
   \Phi^s(T_{\bk_{N_{m-k}}}^{\Xi^{n+k}(\tilde\omega)})}\\
&=\frac{\Phi^s(T_{\bi_l}^{\tilde\omega})\sum_\bj\sum_{N_1}
   \Phi^s(T_{(\bi_l)\bj}^{\tilde\omega})
   \Phi^s(T_{\bj_{N_1}}^{\Xi^{n+k-1}(\tilde\omega)})}
 {c\sum_{N_{n+k-1}}\Phi^s(T_{\bi_{N_{n+k-1}}}^{\tilde\omega})}.
\end{align*}
Recall that in every family there are at most $M$ maps, $\Phi^s(T_j)\le 1$ for
all $j$ and $k\le\lceil\varepsilon n\rceil$, and suppose that
$\varepsilon<p^{\tilde\omega}(s)$. Applying \eqref{neckbound} in the numerator
and \eqref{pressurebound} in the denominator, we obtain for all
$l>N_{n_0(\tilde\omega,\varepsilon)}$ that
\[
\mu_{n+m}^{\tilde\omega}([\bi_l])\le\frac{\Phi^s(T_{\bi_l}^{\tilde\omega})
   M^{N_n(\tilde\omega)-N_{n-1}(\tilde\omega)+N_{n+\lceil\varepsilon n\rceil}(
   \tilde\omega)-N_n(\tilde\omega)}}
  {c e^{(p^{\tilde\omega}(s)-\varepsilon)N_{n+k-1}(\tilde\omega)}}
 \le\frac{\Phi^s(T_{\bi_l}^{\tilde\omega})M^{2\varepsilon N_n(\tilde\omega)}}
  {c e^{(p^{\tilde\omega}(s)-\varepsilon)N_n(\tilde\omega)}}.
\]
Taking $\varepsilon$ so small that
$M^{2\varepsilon}<e^{p^{\tilde\omega}(s)-\varepsilon}$, we set
\[
D(\tilde\omega)=\max\Bigl\{c^{-1},\max_{l\le N_{n_0(\tilde\omega,\varepsilon)}}
  \Bigl\{\frac{\mu^{\tilde\omega}[\bi_l]}{\Phi^s(T_{\bi_l}^{\tilde\omega})}
  \Bigr\}\Bigr\}.
\]
Then for all $l>0$ we have
\[
\mu_{n+m}^{\tilde\omega}([\bi_l])\le D(\tilde\omega)
   \Phi^s(T_{\bi_l}^{\tilde\omega}).
\]
Letting $m$ tend to infinity and recalling that cylinders are open, we
obtain \eqref{eq66} from the Portmanteau theorem \cite[Theorem 17.20]{K}.
\end{proof}

\end{document}